\documentclass[11pt]{amsart}

\usepackage{amsmath,amssymb,amsthm}
\usepackage[pdftex]{color,graphicx}
\usepackage[utf8]{inputenc} 
\usepackage{graphics}
\usepackage{enumerate}
\usepackage{array}
\usepackage[english]{babel}
\usepackage{verbatim}
\usepackage{hyperref}
\usepackage{caption}
\usepackage{tikz-cd}

\usepackage{pdflscape}
\usepackage{rotating}
\usepackage{graphicx}

\usepackage{amsmath}
\usepackage{amsfonts}
\usepackage{amssymb}
\usepackage{amsthm}
\usepackage{mathrsfs}
\usepackage{mathtext}
\usepackage{mathtools}

\usepackage[all,knot]{xy}

\newcommand{\N}{\mathbb{N}}

\newcommand{\tens}{\otimes}
\newcommand{\ot}{\otimes}

\DeclareMathOperator{\HH}{HH}

\DeclareMathOperator{\Img}{Im}
\DeclareMathOperator{\Ker}{Ker}

\newcommand{\bu}{\setlength{\unitlength}{1pt}\begin{picture}(2.5,2)
               (1,1)\put(2,2.5){\circle*{2}}\end{picture}}

\numberwithin{equation}{section}

\newtheorem{defi}[equation]{Definition}
\newtheorem{rema}[equation]{Remark}
\newtheorem{theo}[equation]{Theorem}

\newtheorem{lemm}[equation]{Lemma}
\newtheorem{lemma}[equation]{Lemma}

\title[Gerstenhaber brackets of general twisted tensor products]{Gerstenhaber brackets on Hochschild cohomology \\ 
of general twisted tensor products}

\author[T. Karada\u{g}]{Tekin Karada\u{g}}
\address{Department of Mathematics, Texas A\&M University, 
College Station, Texas 77843, USA}
\email{tekinkaradag@math.tamu.edu}
\author[D. McPhate]{Dustin McPhate}
\address{Department of Mathematics, Texas A\&M University, 
College Station, Texas 77843, USA}
\email{dmcphate@math.tamu.edu}
\author[P.\ S.\ Ocal]{Pablo S.\ Ocal}
\address{Department of Mathematics, Texas A\&M University, 
College Station, Texas 77843, USA}
\email{pso@math.tamu.edu}
\author[T.\ Oke]{Tolulope Oke}
\address{Department of Mathematics, Texas A\&M University, 
College Station, Texas 77843, USA}
\email{toluoke@math.tamu.edu}
\author[S.\ Witherspoon]{Sarah Witherspoon}
\address{Department of Mathematics, Texas A\&M University, 
College Station, Texas 77843, USA}
\email{sjw@math.tamu.edu}

\date{July 22, 2024}

\thanks{Partially supported by NSF grant 1665286.}


\begin{document}


\begin{abstract}
We present techniques for computing Gerstenhaber brackets
on Hochschild cohomology of general twisted tensor product algebras. 
These techniques involve twisted tensor product resolutions
and are based on recent results on Gerstenhaber brackets
expressed on arbitrary bimodule resolutions.
\end{abstract}

\maketitle

\section{Introduction}\label{sec:introduction}

\v{C}ap, Schichl, and Van\v{z}ura~\cite{TTPA} observed that  
whenever an algebra over a field has underlying vector space
given by a tensor product of two subalgebras, 
it takes the form of a twisted tensor product algebra.
Multiplication is given by a twisting map that determines how to move elements of one subalgebra past the other.
In the special case that the twisting is given by a bicharacter,
techniques were developed in~\cite{GNW}
for computing the Lie algebra structure on the Hochschild
cohomology of a twisted tensor product algebra. 
Here we generalize these techniques, describing how to compute
Gerstenhaber brackets for a general
twisted tensor product algebra 
on a twisted tensor product resolution
from~\cite{RTTP} based on the computational methods of~\cite{NW}.
We illustrate this with the Jordan plane, an 
example that was computed by Lopes and Solotar~\cite{LS}
using completely different methods.
A special case of a twisted tensor product algebra is a skew
group algebra, and there is a parallel 
development of Gerstenhaber bracket techniques for
skew group algebras in~\cite{SW19}. 
We generalize some of those results here, pointing out
additional necessary conditions in the general case.
Gerstenhaber brackets are notoriously difficult to compute,
and having a variety of techniques at hand is important.
Our results here add to the collection of techniques available. 

The contents of the paper are as follows:
We define general twisted tensor product algebras $A\ot_{\tau}B$
in Section~\ref{sec:twisted}
and describe the construction of 
twisted tensor product resolutions
of $A\ot_{\tau}B$-bimodules~\cite{RTTP}.
We then recall from~\cite{GG,RTTP} the case of bar 
and Koszul resolutions.  
Our main results are in Section~\ref{sec:Gbracket-techniques} 
where we prove that under some conditions, Gerstenhaber brackets can be computed on
twisted tensor product resolutions using techniques from~\cite{NW};
see Theorem~\ref{mainthm}, which generalizes results in~\cite{GNW,SW19}. 
We illustrate this with the Jordan plane in 
Section~\ref{sec:Jordan}.

\section{Twisted tensor product algebras and resolutions}\label{sec:twisted}

In this section, we recall the notions of general twisted tensor product
algebras from~\cite{TTPA} and their twisted tensor product resolutions 
from~\cite{RTTP}. 
More details may be found in those papers. 

Let $k$ be a field and denote $\otimes := \otimes_k$.
Let $A,B$ be $k$-algebras and 
let $A^e$, $B^e$ be their respective enveloping algebras
($A^e=A\ot A^{op}$, $B^e=B\ot B^{op}$). A bijective $k$-linear map $\tau:B\otimes A\longrightarrow A\otimes B$ is called a \emph{twisting map} if $\tau(1_B\otimes a) = a\otimes 1_B$ and $\tau(b \otimes 1_A) = 1_A \otimes b$ for all $a\in A$, $b\in B$
where $1_A, 1_B$ are the multiplicative identities of $A,B$, respectively,
and the following diagram commutes:
\begin{equation}\label{def:tau}
\xymatrix{B\otimes B\otimes A\otimes A \ar[r]^-{m_B\otimes m_A} \ar[d]_-{1\otimes\tau\otimes 1} & B\otimes A \ar[r]^-{\tau} & A\otimes B\\
B\otimes A\otimes B\otimes A \ar[r]^-{\tau\otimes \tau} & A\otimes B\otimes A\otimes B \ar[r]^-{1\otimes \tau\otimes 1} & A\otimes A\otimes B\otimes B \ar[u]_-{m_A\otimes m_B}}
\end{equation}
The \emph{twisted tensor product algebra}, denoted $A\otimes_{\tau} B$, is the vector space $A\otimes B$ with multiplication:
\begin{equation*}
\xymatrix{m_{\tau}: (A\otimes B)\otimes(A\otimes B) \ar[rr]^-{1\otimes \tau\otimes 1} && A\otimes A\otimes B\otimes B \ar[rr]^-{m_A\otimes m_B} && A\otimes B} .
\end{equation*}
This multiplication is associative provided that $\tau$ is a twisting map. 
The inverse map $\tau^{-1}:A\ot B\rightarrow B\ot A$ satisfies analogous
conditions, and there is an isomorphism of algebras
$A\ot_{\tau} B\stackrel{\sim}{\longrightarrow} 
B\ot_{\tau^{-1}} A$ given by $\tau^{-1}$. 

We will often work with algebras graded by the natural numbers:
Suppose that $A = \oplus_{n\in \N} A_n$ and $B =\oplus_{n\in \N}B_n$,
where $\N$ is understood to contain~0 and $1_A\in A_0$, $1_B\in B_0$.
Then the vector space $A\ot B$ is graded with 
$(A\ot B)_n = \oplus_{i+j=n}A_i\ot B_j$ for all $n$, and $B\ot A$
is graded similarly. 
We say that a twisting map $\tau$ is {\em graded} if
$\tau((B\ot A)_n)= (A\ot B)_n$ for all $n$, and that it is
{\em strongly graded} if $\tau (B_j\ot A_i)=  A_i\ot B_j$
for all $i,j$.

An $A$-bimodule $M$, whose bimodule structure is given by $\rho_A:A\otimes M\otimes A\longrightarrow M$, is said to be \emph{compatible with $\tau$} if there exists a bijective $k$-linear map $\tau_{B,M}:B\otimes M\longrightarrow M\otimes B$ such that the following diagram commutes:
\begin{equation*}
\xymatrix{B\otimes B\otimes M \ar[r]^-{1\otimes \tau_{B,M}} \ar[d]_-{m_B\otimes 1} & B\otimes M\otimes B \ar[r]^-{\tau_{B,M}\otimes 1} & M\otimes B\otimes B \ar[d]^-{1\otimes m_B}\\
B\otimes M \ar[rr]^-{\tau_{B,M}} & & M\otimes B\\
B\otimes A\otimes M\otimes A \ar[u]^-{1\otimes \rho_A} \ar[d]_-{\tau\otimes 1\otimes 1} & & A\otimes M\otimes A\otimes B \ar[u]_-{\rho_A\otimes 1}\\
A\otimes B\otimes M\otimes A \ar[rr]^-{1\otimes \tau_{B,M}\otimes 1} & & A\otimes M\otimes B\otimes A \ar[u]_-{1\otimes 1\otimes \tau}}
\end{equation*}
We can analogously define how a $B$-bimodule $N$ is compatible with $\tau$ via a bijective $k$-linear map $\tau_{N,A}: N\ot A\rightarrow A\ot N$.

Note that the $A^e$-module $M=A$ is itself compatible with $\tau$,
taking $\tau_{B,A}$ to be simply the map $\tau$:
We can see that the upper part of the diagram above commutes by considering
diagram~(\ref{def:tau}) applied to elements for which the rightmost
tensor factor is $1_A$;
the bottom part commutes by iterating the commutativity expressed
by diagram~(\ref{def:tau}) and taking the leftmost tensor factor to be $1_B$.

We next define an $A\ot_{\tau}B$-bimodule structure on the 
tensor product of compatible bimodules. 

\begin{defi}\label{rem:bimod-str} {\em 
$\mathbf{Bimodule}\, \,  \mathbf{Structure} : $
Let $M$ and $N$ be $A$- and $B$-bimodules via $\rho_A:A\ot M\ot A \rightarrow M$
and $\rho_B:B\ot N\ot B\rightarrow N$, respectively. 
Assume that $M$, $N$ are compatible with 
$\tau$ via $\tau_{B,M}$, $\tau_{N,A}$.
Then $M\otimes N$ has a natural structure of an $A\otimes_{\tau} B$-bimodule via $\rho_{A\otimes_{\tau} B}$, given by the following commutative diagram:
\begin{equation*}
\xymatrix{
(A\otimes_{\tau} B)\otimes (M\otimes N) \otimes (A\otimes_{\tau} B) \ar[d]_-{1\otimes\tau_{B,M}\otimes \tau_{N,A}\otimes 1} \ar[rr]^-{\rho_{A\otimes_{\tau} B}} && M\otimes N\\
A\otimes M\otimes B\otimes A\otimes N\otimes B \ar[rr]^-{1\otimes1\otimes\tau \otimes 1\otimes 1} && A\otimes M\otimes A\otimes B\otimes N\otimes B \ar[u]_-{\rho_A\otimes\rho_B}}
\end{equation*} }
\end{defi}

Next we recall the 
notion of compatibility of resolutions. Let $P_{\bu}(M)$ be an $A^e$-projective resolution of $M$ and $P_{\bu}(N)$ a $B^e$-projective resolution of $N$:
\begin{align*}
\cdots \longrightarrow P_2(M) \longrightarrow P_1(M) \longrightarrow P_0(M) \longrightarrow M \longrightarrow 0,\\
\cdots \longrightarrow P_2(N) \longrightarrow P_1(N) \longrightarrow P_0(N) \longrightarrow N \longrightarrow 0.
\end{align*}
We consider the complexes $P_{\bu}(N)\otimes A$, $A\otimes P_{\bu}(N)$, $P_{\bu}(M)\otimes B$, $B\otimes P_{\bu}(M)$. As exact sequences of vector spaces, note that the $k$-linear maps $\tau_{N,A}:N\otimes A\longrightarrow A\otimes N$ and $\tau_{B,M}:B\otimes M\longrightarrow M\otimes B$ can be lifted to $k$-linear chain maps:
\begin{equation*}
\tau_{P_{\bu}(N),A}:P_{\bu}(N)\otimes A\longrightarrow A\otimes P_{\bu}(N),\quad \tau_{B,P_{\bu}(M)}:B\otimes P_{\bu}(M)\longrightarrow P_{\bu}(M)\otimes B,
\end{equation*}
which we will denote more simply by $\tau_{i,A} := \tau_{P_{i}(N),A}$ and $\tau_{B,i} := \tau_{B,P_{i}(M)}$. Given $M$ an $A$-bimodule that is compatible with $\tau$, we say that a projective $A^e$-resolution $P_{\bu}(M)$ is \emph{compatible with $\tau$} if each $P_{i}(M)$ is compatible with $\tau$ via a map $\tau_{B,i}:B\otimes P_{i}(M)\longrightarrow P_{i}(M)\otimes B$ such that $\tau_{B,\bu}$ is a chain map lifting $\tau_{B,M}$. Given a $B$-bimodule $N$ compatible with $\tau$ we can analogously define how a projective $B^e$-resolution $P_{\bu}(N)$ is compatible with $\tau$ via $\tau_{\bu,A}$.

Provided the resolutions $P_{\bu}(M)$, $P_{\bu}(N)$ are compatible with $\tau$, 
by~\cite[Lemma 3.5]{RTTP}, the total complex of the 
tensor product complex $P_{\bu}(M)\ot P_{\bu}(N)$
has homology concentrated in degree~0, where it is $M\ot N$.
If in addition, 
each $(A\ot_{\tau} B)^e$-module $P_i(M)\ot P_j(N)$ is projective,
it follows immediately that this complex is a projective resolution,
called a {\em twisted tensor product resolution}:

\begin{theo}\label{thm:ttpr}
Let $M,N$ be $A$- and $B$-bimodules that are compatible with
a twisting map $\tau: B\ot A\rightarrow A\ot B$.
Let $P_{\bu}(M)$, $P_{\bu}(N)$ be projective
$A^e$- and $B^e$-module resolutions of $M$ and $N$, respectively, that are compatible
with $\tau$.
If each $P_i(M)\ot P_j(N)$ is a projective $(A\ot_{\tau} B)^e$-module 
under the
module structure given  in Definition~\ref{rem:bimod-str}, then
the total complex of  $P_{\bu}(M)\ot
P_{\bu}(N)$ is a projective $(A\ot_{\tau} B)^e$-module resolution
of $M\ot N$.
\end{theo}
 
We will see directly that the hypotheses of the theorem are true of
all the resolutions that we consider in this paper.
Alternatively, see~\cite[Theorem 3.10]{RTTP} for additional hypotheses
ensuring that all modules $P_i(M)\ot P_i(N)$ are projective.
Next, we explain in detail the special cases where $M=A$,
$N=B$, and $P_{\bu}(A)$, $P_{\bu}(B)$ are either bar resolutions
or Koszul resolutions (when $A$, $B$ are Koszul algebras).
We will then be able to use these twisted tensor product resolutions
to compute the Hochschild cohomology $\HH^*(A\ot_{\tau}B)$ as
${\rm{Ext}}^*_{(A\ot_{\tau}B)^e}(A\ot_{\tau}B, A\ot_{\tau}B)$.

Consider the sequence of left $A^e$-modules (equivalently $A$-bimodules):
\begin{equation*}
\cdots \stackrel{d_3}{\longrightarrow} A^{\otimes 4} \stackrel{d_2}{\longrightarrow} A^{\otimes 3} \stackrel{d_1}{\longrightarrow} A^{\otimes 2} \stackrel{m_A}{\longrightarrow} A \longrightarrow 0,
\end{equation*}
with differentials, for all $n\geq 1$:
\begin{equation}\label{def:dif}
d_n(a_0\otimes\cdots\otimes a_{n+1}) = \sum_{i=0}^n (-1)^i a_0\otimes\cdots\otimes a_ia_{i+1}\otimes\cdots \otimes a_{n+1}.
\end{equation}
The \emph{bar resolution} $\mathbb{B}(A)$ is the truncated complex, that is, $\mathbb{B}_n(A) = A^{\otimes (n+2)}$ for all $n\geq 0$. Since $k$ is a field, $\mathbb{B}(A)$ is a free left $A^e$-module resolution of $A$. The module structure can be expressed as $\rho_{A,A^{\otimes n}}: A\otimes A^{\otimes n}\otimes A \longrightarrow A^{\otimes n}$ where $\rho_{A,A^{\otimes n}}(a\otimes c_1\otimes \cdots \otimes c_n\otimes b) = a c_1\otimes \cdots \otimes c_n b$ for all $a,b,c_1,\ldots,c_n$, that is $\rho_{A,A} = m_A (1\otimes m_A)$ and $\rho_{A,A^{\otimes n}} = m_A\otimes 1_{n-2} \otimes m_A$ for $n\geq 2$, where $1_{n-2}$ denotes the identity map on $A^{\ot (n-2)}$.

The bar resolution of $A$ is compatible with $\tau$ 
(see, e.g.,~\cite{GG} or~\cite[Proposition 2.20(i)]{RTTP}):

\begin{theo}\label{thm:bar-compatible}
Let $A,B$ be $k$-algebras and 
let $\tau :B\ot A\rightarrow A\ot B$ be a twisting map.
The bar resolutions $\mathbb{B}(A)$ and $\mathbb{B}(B)$
are compatible with $\tau$.
\end{theo}

The proof, as described in~\cite{RTTP}, is iteration of the commutativity
of diagram~(\ref{def:tau}).

Now let $V$ be a finite dimensional vector space over $k$ and let $T(V) = \oplus_{n\in \mathbb{N}}T^n(V)$ denote the tensor algebra of $V$, so that $T^n(V) = V\otimes\cdots\otimes V$ ($n$ tensor factors) for each $n\in\mathbb{N}$. 
Then $T(V)$ may be considered to be a graded algebra, where for any $v\in V$ we use $|v|$ to denote its degree and assign $|v| = 1$.

Given a vector subspace $R\subseteq V\otimes V = T^2(V)$, set $A = T(V)/(R)$ where $(R)$ denotes the ideal generated by $R$ in $T(V)$. 
Then $A$ is a 
graded algebra generated by elements in degree $1$ with relations $R$ in degree $2$. Consider the sequence of $A$-bimodules:
\begin{equation*}
\cdots \stackrel{d_4}{\longrightarrow} \mathbb{K}_3(A) \stackrel{d_3}{\longrightarrow} A\otimes R\otimes A \stackrel{d_2}{\longrightarrow} A\otimes V\otimes A \stackrel{d_1}{\longrightarrow} A\otimes A \stackrel{m_A}{\longrightarrow} A \longrightarrow 0,
\end{equation*}
where $\mathbb{K}_0(A) = A\otimes A$, $\ \mathbb{K}_1(A) = A\otimes V\otimes A$, $\ \mathbb{K}_2(A) = A\otimes R\otimes A$ and 
\[
   \mathbb{K}_n(A) = A\otimes \bigcap_{i=0}^{n-2}\left(V^{\otimes i}
        \otimes R\otimes V^{\otimes (n-i-2)} \right)\otimes A
\]
for $n\geq 3$, with differentials given by \eqref{def:dif} (that is, the same differentials as in the bar resolution, viewing each $\mathbb{K}_n(A)$ as a subset
of $A^{\ot (n+2)}$).
Since $m_A(R)=0$, the image of $\mathbb{K}_n(A)$ under $d_n$ is
indeed contained in $\mathbb{K}_{n-1}(A)$.
We see in this way that $\mathbb{K}(A)$
is a subcomplex of the bar resolution. 

We say that $A$ is a \textit{Koszul algebra} whenever the truncated complex $\mathbb{K}(A)$ is a resolution of $A$ as an $A^e$-module. 
This definition of Koszul algebra is equivalent to
other definitions in the literature
(see~\cite[Proposition 19]{Kra} or~\cite[Theorem 3.4.6]{HCSW}). 
In this case, we say that $\mathbb{K}(A)$ is the \textit{Koszul resolution} of $A$. 

Koszul resolutions are compatible with any strongly graded
twisting map $\tau$~\cite{GG,RTTP}.
This follows for example by using  
techniques in the proof of~\cite[Proposition 1.8]{Walton-W2} (see also 
\cite[Corollary 4.19]{JLS} or \cite[p.\ 90, Example 3]{PP}).
The idea is that 
the above compatibility statements for the bar resolution
descend to the Koszul resolution.
Specifically, we have the following statement,
which is essentially~\cite[Proposition 2.20(iii)]{RTTP}.
We give just a sketch of a proof here, for completeness.

\begin{theo}\label{thm:Koszul-compatible}
Let $A$ be a Koszul algebra, let $B$ be a graded algebra,
and let $\tau:B\otimes A\rightarrow A\otimes B$
be a strongly graded twisting map.
Then the Koszul resolution $\mathbb{K}(A)$ is compatible with $\tau$.
If $B$ is also a Koszul algebra, 
its Koszul resolution $\mathbb{K}(B)$ is compatible with $\tau$. 
\end{theo}

\begin{proof}
Define a map $\tau_n : B\ot A^{\ot (n+2)}\rightarrow A^{\ot (n+2)}\ot B$ by iterating
$\tau$: First let $\tau_0 = (1\ot \tau)(\tau\ot 1)$, then let
$\tau_1 = (1\ot 1\ot \tau)(1\ot \tau\ot 1)(\tau\ot 1\ot 1)$, 
and so on.
We claim that the subcomplex ${\mathbb{K}}(A)$ of the bar resolution is
preserved by $\tau_{\bu}$ in the following sense.
As $\tau$ is strongly graded, $\tau(B\ot V)= V\ot B$,
so $\tau_1(B\ot A \ot V\ot A)= A\ot V\ot A\ot B$.
Since $\tau$ is a twisting map and $m_A(R)=0$, we find that 
$\tau_2(B\ot A\ot R\ot A)= A\ot R\ot A \ot B$.
In general, 
\[
   \tau_n(B\ot \mathbb{K}_n(A))= \mathbb{K}_n(A)\ot B 
\]
for all $n$.
Note that the map $\tau_{\bu}$ is a chain map since $\tau$ is a twisting
map and the differential is given by~(\ref{def:dif}).
The first statement now follows from Theorem~\ref{thm:bar-compatible},
by restricting $\tau_{\bu}$ to ${\mathbb{K}}(A)$.
The second statement is similar.
\end{proof}

\section{Gerstenhaber bracket for twisted tensor products}\label{sec:Gbracket-techniques}
In this section, we begin by summarizing techniques from~\cite{NW}
for computing Gerstenhaber brackets on the Hochschild
cohomology ring $\HH^*(R)$ of a $k$-algebra $R$, as reformulated 
in~\cite[Section 6.4]{HCSW}.
We will then take $R$ to be a twisted tensor product algebra $A\ot_{\tau}B$
and develop further techniques for handling a twisted tensor product
algebra specifically, generalizing work in~\cite{GNW,SW19}. 

The graded Lie algebra structure on the Hochschild cohomology ring $\HH^*(R)$
was historically 
defined on the bar resolution, with equivalent recent definitions on
other resolutions such as in~\cite{NW,Volkov}. 
Here we will simply take a formula from~\cite{NW}, 
stated in Theorem~\ref{thm:bracket} below,
to be our definition of the Gerstenhaber bracket, 
and refer to the
cited literature for proof that it is equivalent to the historical definition. 

Let $K$ be a projective resolution of $R$ as an $R^e$-module, with
differential $d$ and augmentation map $\mu : K_0\rightarrow R$.
Assume that $(K,d)$ is a 
{\em counital differential graded coalgebra}, i.e.~there is a 
diagonal map $\Delta:K\rightarrow K\ot_R K$
(a chain map lifting $R\stackrel{\sim}{\longrightarrow} R\ot_R R$)
 such that
$(\Delta\ot 1) \Delta = (1\ot \Delta)\Delta$ and $(\mu\ot 1)\Delta
= 1 = (1\ot \mu)\Delta$
where we take $\mu$ to be the zero map on $K_i$ for all $i>0$.

Here and elsewhere, viewing $K$ as a graded vector space,
we adopt the following standard sign convention for tensor
products of maps:
If $V$, $W$, $V'$, $W'$ are graded vector spaces and $g: V\rightarrow V'$,
$h: W\rightarrow W'$ are graded $k$-linear maps, define
\[
   (g\ot h) (v\ot w) = (-1)^{|h||v|} g(v)\ot h(w)
\]
for all homogeneous $v\in V$, $w\in W$, where $|h|$, $|v|$
denote the degrees of $h$, $v$, respectively. 

It can be shown (see, e.g.,~\cite[Section~6.4]{HCSW}) that there exists an
$R$-bimodule map $\phi: K\ot_R K\rightarrow K[1]$ for which 
\begin{equation}\label{eqn:phi-htpy}
   d\phi + \phi (d\ot 1 + 1\ot d) = \mu\ot 1 - 1\ot \mu ,
\end{equation}
where we take $K[1]_i = K_{i+1}$ for all $i$ so that
$\phi$ is a map of degree 1. 
The map $\phi$ is called a {\em contracting homotopy} for $\mu\ot 1 - 1\ot \mu$. 

Let $f\in{\rm{Hom}}_{R^e}(K_m,R)$ be a cocycle, that is $fd_{m+1}=0$. 
We can view it as a map $f:K\rightarrow R$ by setting $f|_{K_n}=0$ whenever $n\neq m$. 
Define $\psi_f : K\rightarrow K[1-m]$ by
\begin{equation}\label{eqn:psi-f}
  \psi_f = \phi (1\ot f\ot 1) \Delta^{(2)} ,
\end{equation}
where $\Delta^{(2)}:= (\Delta\ot 1)\Delta$
and $\phi$ is a contracting homotopy for $\mu\ot 1 - 1\ot\mu$
as defined above. 
In the expression~(\ref{eqn:psi-f}), 
$1\ot f\ot 1$ is a map from $K\ot_R K\ot_R K$
to $K\ot_R R\ot_R K\cong K\ot_R K$, and we have
identified these latter two isomorphic vector spaces. 
Note that $\psi_f$ is a map of degree $1-m$ on $K$ by definition.

The following theorem is from~\cite{NW} with slightly different 
hypotheses, and is
stated in this form as~\cite[Theorem 6.4.5]{HCSW}.
We will take the formula in the theorem for the Gerstenhaber bracket
as a definition for our purposes in the rest of
this paper, since our resolutions will always be counital
differential graded coalgebras. 

\begin{theo}\label{thm:bracket}
Let $K$ be a projective resolution of $R$ as an $R^e$-module
that is a counital differential graded coalgebra. 
Let $f\in {\rm{Hom}}_{R^e}(K_m,R)$ and $g\in{\rm{Hom}}_{R^e}(K_n,R)$ be cocycles.
Define $\psi_f$ as in~(\ref{eqn:psi-f}), and similarly $\psi_g$.
Then
\begin{equation}\label{eqn:formula-bracket}
   {[} f , g {]} := f\psi_g - (-1)^{(m-1)(n-1)} g \psi_f ,
\end{equation}
as a function in ${\rm{Hom}}_{R^e}(K_{m+n-1},R)$, represents the
Gerstenhaber bracket on Hochschild cohomology. 
\end{theo}

We next explain how to obtain a map $\phi$
satisfying equation~(\ref{eqn:phi-htpy}) when $K$ is a twisted tensor
product resolution.
We may then use $\phi$ in the formula~(\ref{eqn:psi-f})
in order to compute Gerstenhaber brackets via formula~(\ref{eqn:formula-bracket}).

Let $A$ and $B$ be $k$-algebras and let $\tau: B\ot A\rightarrow A\ot B$
be a twisting map. Let 
\begin{align*}
 P_{\bu}: &\quad  \cdots \xrightarrow{d_{2}^{P}}P_{1}\xrightarrow{d_{1}^{P}}P_{0}\xrightarrow{\mu_P}A\to 0 ,\\
Q_{\bu}:  &\quad \cdots \xrightarrow{d_{2}^{Q}}Q_{1}\xrightarrow{d_{1}^{Q}}Q_{0}\xrightarrow{\mu_Q}B\to 0 
\end{align*}
be an $A^{e}$-projective resolution of $A$ and 
a $B^{e}$-projective resolution of $B$, respectively.
By Theorem~\ref{thm:ttpr}, if $P_{\bu}$ and $Q_{\bu}$ 
are compatible with $\tau$
and $P_i\ot Q_j$ is a projective $(A\ot_{\tau}B)^e$-module for all $i,j$,
then the total complex of 
$P_{\bu}\tens Q_{\bu}$ is an $(A\tens_{\tau}B)^{e}$-projective resolution of $A\tens_{\tau}B$.
We will denote this resolution by $P_{\bu}\tens_{\tau} Q_{\bu}$. 
In particular, this is the case when $P$ and $Q$ are both bar resolutions
or they are both Koszul resolutions and $\tau$ is strongly graded, 
by Theorems~\ref{thm:bar-compatible} and~\ref{thm:Koszul-compatible}. 

We will be interested in resolutions $P_{\bu},Q_{\bu}$ for which there exists
an $A\ot_{\tau}B$-bimodule chain map 
\begin{equation}\label{eqn:sigma-domain-codomain}
   \sigma: (P_{\bu}\ot_{\tau}Q_{\bu})\ot _{A\ot_{\tau}B} 
   (P_{\bu}\ot_{\tau}Q_{\bu})
   \rightarrow (P_{\bu}\ot_A P_{\bu})\ot_{\tau}(Q_{\bu}\ot_BQ_{\bu})
\end{equation}
such that 
\begin{equation}
\label{eqn:sigma-condn}
    \mu_P\ot\mu_Q\ot 1_P\ot 1_Q - 1_P\ot 1_Q\ot \mu_P\ot \mu_Q
   = (\mu_P\ot 1_P\ot \mu_Q\ot 1_Q - 1_P\ot \mu_P\ot 1_Q\ot \mu_Q)\sigma 
\end{equation}
as a map from $(P_{\bu}\ot_{\tau}Q_{\bu})\ot_{A\ot_{\tau}B} 
(P_{\bu}\ot_{\tau}Q_{\bu})$
to $P_{\bu}\ot_{\tau}Q_{\bu}$.
We will see that such a map $\sigma$ exists for bar and Koszul
resolutions in particular in Lemma~\ref{a} below.

First we define a map that will be a forerunner to $\sigma$
for these resolutions: Assume for now that $P_{\bu}$ and $Q_{\bu}$ 
are both bar resolutions.
Let $\sigma': P_{\bu}\ot Q_{\bu}\ot P_{\bu}\ot Q_{\bu} 
\rightarrow P_{\bu}\ot P_{\bu}\ot Q_{\bu}\ot Q_{\bu}$ 
be the map given in each degree as
$
\sigma'_{r,s,t,u}:P_r\otimes Q_s\otimes P_t\otimes Q_u\longrightarrow P_r\otimes P_t\otimes Q_s\otimes Q_u
$
for $r,s,t,u\in\mathbb{N}$ where
\begin{align}\label{eqn:sigma-defn}
\sigma'_{r,s,t,u} = &(-1)^{st} (1_{r+2}\otimes\tau_{B,t}\otimes 1_{s+1}\otimes 1_{u+2})\circ (1_{r+2}\otimes 1\otimes\tau_{B,t}\otimes 1_s\otimes 1_{u+2})\\
&\circ\cdots\circ(1_{r+2}\otimes 1_s\otimes\tau_{B,t}\otimes 1\otimes 1_{u+2}) \circ (1_{r+2}\otimes 1_{s+1}\otimes\tau_{B,t}\otimes 1_{u+2}).\nonumber
\end{align}
This definition says that $\sigma'$ is the map that takes the rightmost element in the tensor product $Q_s$, passes it through $P_t$ via $\tau_{B,t}$, then takes the second rightmost element and proceeds likewise, and so on, till we have passed  factors in $Q_s$ to the right side of factors in $P_t$. However, we could proceed in a symmetric way by first taking the leftmost element in the tensor product of $P_t$, passing it through $Q_s$ via $\tau_{s,A}$, then proceeding analogously as before, till we have passed all elements forming $P_t$ to the left side of $Q_s$. By properties of the twisting map $\tau$, these two constructions will be the same.

Next assume that $P_{\bu}$ and $Q_{\bu}$ are Koszul resolutions and $\tau$ is strongly
graded.
An argument similar to the proof of Theorem~\ref{thm:Koszul-compatible} shows
that $\sigma'$, defined similarly by equation~(\ref{eqn:sigma-defn}) in each
degree, is indeed a well-defined map on Koszul resolutions.

\begin{lemma}\label{a}
Let $P_{\bu},Q_{\bu}$ be bar resolutions or Koszul resolutions in case $\tau$ is strongly graded
and define $\sigma'$ as in equation~(\ref{eqn:sigma-defn}). Then $\sigma'$ induces a chain map $\sigma$
as in~(\ref{eqn:sigma-domain-codomain}) that is an isomorphism of 
$(A\tens_{\tau}B)^{e}$-modules in each degree, 
lifting the canonical isomorphism 
$$
   (A\ot_{\tau} B)\ot_{A\ot_{\tau}B} (A\ot_{\tau} B)\stackrel{\sim}{\longrightarrow}
    A\ot_{\tau} B .
$$
In particular, condition~(\ref{eqn:sigma-condn}) holds.
\end{lemma}
 
\begin{proof}
First we note that $\sigma'$, defined by 
equation~(\ref{eqn:sigma-defn}) in each degree, induces a map 
\[
   \sigma :  (P_{\bu}\ot_{\tau}Q_{\bu})\ot _{A\ot_{\tau}B} 
   (P_{\bu}\ot_{\tau}Q_{\bu})
   \rightarrow (P_{\bu}\ot_A P_{\bu})\ot_{\tau}(Q_{\bu}\ot_BQ_{\bu})
\]
as claimed: Similarly to the proof of~\cite[Lemma~4.1]{SW19}, the map
given by the composition
\[
  (P_{\bu}\ot Q_{\bu})\ot (P_{\bu}\ot Q_{\bu}) 
  \stackrel{1\ot \tau\ot 1}{\relbar\joinrel\relbar\joinrel\relbar\joinrel
\longrightarrow} P_{\bu}\ot P_{\bu}\ot Q_{\bu}\ot Q_{\bu} \longrightarrow
(P_{\bu}\ot_{A} P_{\bu})\ot (Q_{\bu}\ot_{B} Q_{\bu}),
\]
where the latter map is the canonical surjection,
is $A\ot_{\tau}B$-middle linear.
Therefore it induces a well-defined map $\sigma$ as claimed. 
Further, $\sigma$ is a chain map 
since $\tau$ is compatible with the multiplication maps
$\mu_A$, $\mu_B$, the bimodule actions are defined in terms of
multiplication, and Koszul resolutions are subcomplexes of bar resolutions.
Finally, $\sigma$ is a bijection;
an inverse map can be defined similarly using the inverse
of the twisting map $\tau$. 
\end{proof}

Other settings where a map $\sigma$ exists that satisfies condition~(\ref{eqn:sigma-condn})
are skew group algebras~\cite[Remark 4.4]{SW19} and the Jordan plane
of Section~\ref{sec:Jordan} below.
We see next that condition~(\ref{eqn:sigma-condn}) allows us to define
a contracting homotopy, generalizing the cases where the twisting
is given by a bicharacter~\cite[Lemma 3.5]{GNW} or 
by a group action resulting in a skew group algebra~\cite[Theorem 4.6]{SW19}. 
Our proof is essentially the same as in these two cases;
we provide details for completeness. 

\begin{lemm}\label{c}
Let $P_{\bu},Q_{\bu}$ 
be projective $A^e$- and $B^e$-module resolutions, respectively,
both compatible with $\tau$, for which $P_{\bu}\ot_{\tau}Q_{\bu}$ 
is a projective
$(A\ot_{\tau}B)^e$-module resolution of $A\ot_{\tau}B$. 
Let $\phi_{P}$ and $\phi_{Q}$ be contracting homotopies for 
$\mu_P\ot 1_P-1_P\ot \mu_P$ and 
$\mu_Q\ot 1_Q - 1_Q\ot \mu_Q$, respectively. 
Assume there is a chain map $\sigma$ 
as in~(\ref{eqn:sigma-domain-codomain})
satisfying condition~(\ref{eqn:sigma-condn}). 
Define $\phi=\phi_{P\tens_{\tau}Q} : 
(P_{\bu}\tens_{\tau}Q_{\bu})\tens_{A\tens_{\tau}B}
(P_{\bu}\tens_{\tau}Q_{\bu})\to (P_{\bu}\tens_{\tau}Q_{\bu})$ by
\begin{equation}\label{3}
 \phi:=(\phi_{P}\tens\mu_Q\ot 1_Q+1_P\ot \mu_P\tens\ \phi_{Q})\sigma .
\end{equation}
Then $\phi$ is a contracting homotopy for 
$\mu_P\ot \mu_Q\ot 1_P\ot 1_Q - 1_P\ot 1_Q\ot \mu_P\ot \mu_Q$. 
\end{lemm}
\begin{proof}
To shorten notation, as in~\cite{NW}, we define chain maps $F_{P}^{l}, F_{P}^{r}, F_{P} : P\tens_{A}P \to P$ by
\[
 F_{P}^{l}=\mu_P\tens\ 1_{P},
\quad
 F_{P}^{r}=1_{P}\tens\ \mu_P,
\quad
\text{ and }\quad
F_{P}=F_{P}^{l}-F_{P}^{r},
\]
where $1_{P}$ is the identity map on $P$. We define $F_{Q}^{l}, F_{Q}^{r}$,and $F_{Q}$ similarly. As $F_{P}^{l}, F_{Q}^{l}, F_{P}^{r}, F_{Q}^{r}$, and $\sigma$ are chain maps, they commute with the differentials, so for example
$ d_P F^r_P = F^r_P d_{P\ot_A P}$, and similar equations hold for the other maps. 
Therefore by the above formulas, notation, and the definition of the differential on a
tensor product of complexes,
\begin{align*}
 d\phi+\phi d
 =&(d\tens  1 + 1\tens d)(\phi_{P}\tens F_{Q}^l+F_{P}^r\tens \phi_{Q})\sigma \\
 &+(\phi_{P}\tens F_{Q}^l+F_{P}^r\tens \phi_{Q})\sigma(d\tens 1 + 1\tens d) \\
 =&\big(d\phi_{P}\tens F_{Q}^l+dF_{P}^r\tens \phi_{Q}+\phi_{P}\tens dF_{Q}^l+F_{P}^r\tens d\phi_{Q} \\
 &+\phi_{P}d\tens F_{Q}^l+\phi_{P}\tens F_Q^ld+F_P^rd\tens \phi_Q + F_P^r\tens \phi_Qd\big)\sigma\\
 =&((d\phi_P+\phi_Pd)\tens\ F_Q^l+F_P^r\tens\ (d\phi_Q+\phi_Qd))\sigma . 
\end{align*}
Applying (\ref{eqn:phi-htpy})
 to rewrite $d\phi_P+\phi_Pd$ and $d\phi_Q+\phi_Qd$, this
expression is equal to 
\begin{align*}
 (F_P\tens\ F_Q^l+F_P^r\tens\ F_Q)\sigma 
  &=((F_P^l-F_P^r)\tens\ F_Q^l+F_P^r\tens\ (F_Q^l-F_Q^r))\sigma \\
 &=(F_P^l\tens\ F_Q^l-F_P^r\tens\ F_Q^r)\sigma .
\end{align*}
By condition~(\ref{eqn:sigma-condn}), 
this is equal to $\mu_P\ot\mu_Q\ot 1_P\ot 1_Q
- 1_P\ot 1_Q\ot\mu_P\ot \mu_Q$,
as desired.
\end{proof}

As a consequence of Lemma~\ref{c},
we may now define a contracting homotopy $\phi_{P\ot_{\tau}Q}$ from 
knowledge of $\phi_P$, $\phi_Q$ provided there is a chain map $\sigma$
satisfying condition~(\ref{eqn:sigma-condn}).
We may then compute Gerstenhaber brackets
from these maps. We have thus proven:

\begin{theo}
\label{mainthm}
Let $A$ and $B$ be $k$-algebras and let $P_{\bu}$ and $Q_{\bu}$ be a projective
$A^e$- and $B^e$-module resolution of $A$ and $B$, respectively.
Assume that $P_{\bu}\ot_{\tau} Q_{\bu}$ is a projective
$(A\ot_{\tau} B)^e$-module resolution of $A\ot_{\tau}B$ that
is a counital differential graded coalgebra and that there is a
chain map $\sigma$ satisfying condition~(\ref{eqn:sigma-condn}). 
Then Gerstenhaber brackets of Hochschild cocycles on 
$P_{\bu}\ot_{\tau}Q_{\bu}$
are given by formula~(\ref{eqn:formula-bracket}) via
formulas~(\ref{eqn:psi-f}) and~(\ref{3}). 
\end{theo}

By Lemma~\ref{a}, the theorem applies 
whenever $P_{\bu}$, $Q_{\bu}$ are bar resolutions,
or when $P_{\bu}$, $Q_{\bu}$ are Koszul resolutions and $\tau$ is strongly graded.
In the next section, we give an example of yet another setting 
in which a suitable chain map $\sigma$ may be defined, and so
the theorem applies.
As already mentioned, the theorem applies more generally to 
twisted tensor products where the twisting is given by a
bicharacter on grading groups~\cite[Section 3]{GNW} and 
to skew group algebras~\cite[Theorem 4.9]{SW19}.

\section{Jordan plane}\label{sec:Jordan}

In this section, we will illustrate the twisted tensor product techniques of Sections~\ref{sec:twisted} and~\ref{sec:Gbracket-techniques} with a small example. 
Let $A = k[x]$ and $B=k[y]$.
Let $\tau: B\ot A \rightarrow A\ot B$ be defined by
\[
\tau(y\ot x) = x\ot y + x^2\ot 1 ,
\]
extended to $B\ot A$ by requiring $\tau$ to be a twisting map.
Its inverse is given by
\[
   \tau^{-1} (x\ot y) = y\ot x - 1\ot x^2.
\]
There is an isomorphism
\[
  A\ot_{\tau} B\cong k\langle x,y\rangle / (yx - xy - x^2) .
\]
Accordingly, we will write elements of the ring $A\ot_{\tau}B$ as
(noncommutative) polynomials in indeterminates $x$ and $y$,
omitting the tensor product symbol in notation for these elements.
We will identify the algebras $A$ and $B$ with subalgebras
of $k\langle x,y\rangle/ (yx-xy-x^2)$ in this way. 

We call $A\ot_{\tau} B$ the {\em Jordan plane}.
Its Hochschild cohomology and in particular the
Gerstenhaber algebra structure was computed by
Lopes and Solotar~\cite{LS} as one example in a larger class of algebras,
using completely different techniques than ours:
Their resolution was constructed using a method of Chouhy and
Solotar~\cite{CS}, analogous to the Anick resolution,
and their Gerstenhaber brackets were computed using the derivation
operator method of Su\'arez-\'Alvarez~\cite{SA} that applies
specifically to brackets with Hochschild 1-cocycles (that is, derivations). 
See also~\cite[Theorem 4.6]{Sh} for derivations
on the Jordan plane and their Gerstenhaber brackets, and~\cite{BLO}
for a generalization. 
Our approach is based on techniques from
Section~\ref{sec:Gbracket-techniques} and~\cite{NW,RTTP}.  
We begin by  explicitly constructing
a twisted tensor product resolution $P_{\bu}\ot_{\tau}Q_{\bu}$ 
for $A\ot_{\tau}B$. 

Note that $A$ and $B$ are Koszul algebras and $\tau$ is graded 
but not strongly graded. 
Thus Lemma~\ref{a} does not apply. 
Nonetheless we will construct a twisted tensor product resolution with a diagonal map $\Delta$ 
and a map $\sigma$ to which Theorem~\ref{mainthm} applies.
We will then be able to 
compute Gerstenhaber brackets on the Hochschild cohomology ring $\HH^*(A\ot_{\tau}B)$ as stated in Theorem~\ref{mainthm}.

Let $P_{\bu},Q_{\bu}$ be the Koszul resolutions for $A,B$, e.g.~$P_{\bu}$ is given by 
\begin{equation*}
\xymatrix{P_{\bu} : &0 \ar[r] & A\otimes A \ar[rrr]^-{\tilde{d_1}=(x\otimes1-1\otimes x).} &&& A\otimes A \ar[r]^-{m_A} & A \ar[r] & 0 , }
\end{equation*}
where we have identified $A\ot V\ot A$ with $A\ot A$
via the canonical isomorphism ($V$
is a one-dimensional vector space), and similarly $Q_{\bu}$.
We use the notation for twisting maps as before
on Koszul resolutions, under this identification.  
We will explicitly construct the twisted tensor product resolution 
$P_{\bu}\ot_{\tau}Q_{\bu}$ for $A\ot_{\tau} B$.

We will choose some notation to help keep track of homological degrees
of the free modules in the resolutions $P_{\bu},Q_{\bu}$:
For $i\in\{0,1\}$, let $e_i=1\ot 1$, a free generator of $P_i$ as an $A^e$-module. 
Similarly let $e'_i=1\ot 1$, a free generator of $Q_i$ as a $B^e$-module.
Calculations show that 
maps $\tau_{0,A}$ and $\tau_{1,A}$ may be given by 
\[
\tau_{0,A}=(\tau\otimes 1_B) (1_B\otimes \tau), \ \ \ 
\tau_{1,A}(e_1'\ot x) = x\ot e_1' .
\]
We also have
\[  
   \tau^{-1}_{1,A}(x\otimes e_1')=e_1'\otimes x . 
\]
Since $A$ is a free algebra on the generator $x$, $B\ot B$ is
a free $B^e$-module on the generator $1\ot 1$, and $\tau$ is a
twisting map, the above value of $\tau_{1,A}$ on $e_1'\ot x$ is sufficient
to determine all other values, and similarly for $\tau^{-1}_{1,A}$. 
Maps $\tau_{B,0}$ and $\tau_{B,1}$ are given by 
\begin{align*}
\tau_{B,0}&=(1_A\otimes \tau)(\tau\otimes 1_A) ,\\
\tau_{B,1}(y\ot e_1) &= e_1\ot y + xe_1\ot 1 + e_1x\ot 1 .
\end{align*}
We also have
\[
   \tau^{-1}_{B,1}(e_1\otimes y)=y\otimes e_1-1\otimes xe_1-1\otimes e_1x.
\]
The total complex $P_{\bu}\otimes_{\tau} Q_{\bu}$ is
\begin{equation}\label{total}
\xymatrix{
0 \ar[r] & P_1 \otimes Q_1 \ar[r]^-{d_2} & (P_1 \otimes Q_0)\oplus (P_0 \otimes Q_1) \ar[r]^-{d_1} &  P_0 \otimes Q_0 \ar[r] & 0}
\end{equation}
with differentials as follows.
Note first that each $(A \ot_{\tau} B)^e$-module $P_i\ot Q_j$ is 
a free $(A\ot_{\tau}B)^e$-module of rank one.
A calculation shows that
\[
    P_0\ot Q_0 := A\ot A\ot B\ot B \stackrel{1\ot \tau^{-1}\ot 1}
      {\relbar\joinrel\relbar\joinrel\relbar\joinrel\relbar\joinrel\relbar\joinrel\longrightarrow}
   (A\ot_{\tau}B) \ot (A\ot_{\tau}B)^{op}
\]
is an isomorphism of $(A\ot_{\tau}B)^e$-modules. For $(i,j)\in \{(1,1),(1,0),(0,1)\}$, 
we now describe an isomorphism
\[
\varphi : (A\ot_{\tau} B)\ot (A\ot_{\tau} B)^{op} \rightarrow P_i\ot Q_j .
\]
The free $(A\ot_{\tau} B)^e$-module
$(A\ot_{\tau} B )\ot (A\ot_{\tau} B)^{op}$ is of rank one with
free basis element $(1\ot 1)\ot (1\ot 1)$. Define 
$\varphi ( (1\ot 1)\ot (1\ot 1)) \coloneqq e_i\ot e_j'$ in $P_i\ot Q_j$ 
and let $\varphi$ be the $(A\ot_{\tau}B)^e$-module homomorphism 
generated by this choice. It satisfies 
\[
\varphi((x^a\ot y^b) \ot (x^{a'}\ot y^{b'})) \coloneqq 
(x^a\ot y^b)(e_i\ot e_j')(x^{a'}\ot y^{b'}) = 
(x^a\ot y^b)(e_ix^{a'} \ot e_j' y^{b'}).
\]
This equals $x^a e_ix^{a'} \ot y^b e_j' y^{b'} + S$, where
$S$ is a sum of terms whose degree in $y$ is less than $b + b'$ and whose 
total degree is still $a + a' + b + b'$. 
By induction on the degree in $y$, the top term
$x^ae_ix^{a'} \ot y^b e_j'y^{b'}$ of this expression is also in the image
of the map $\varphi$. Therefore $\varphi$ is surjective. 
By construction, $\varphi$ can be expressed via an upper 
triangular matrix in each total polynomial degree, so it is also injective. 
We apply the standard formula
for the differential on a tensor product of complexes as well as
the maps $\tau_{i,A}^{-1}$, $\tau_{B,i}^{-1}$ given above in order to express
elements in the resolution in terms of the free basis elements
$e_0\ot e_0'$, $e_0\ot e_1'$, $e_1\ot e_0'$,
$e_1\ot e_1'$. We calculate as follows (recall that
we write elements of the ring $A\ot_{\tau}B$ as noncommutative
polynomials in $x,y$, omitting the tensor symbol
in notation for these elements):
\begin{align*}
d_1(e_1\otimes e'_0)&=\tilde{d_1}(e_1)\otimes e'_0\\
&=xe_0\otimes e'_0 - e_0x\otimes e'_0=xe_0\otimes e'_0-e_0\otimes e'_0x \\
&=(x\otimes 1-1\otimes x)(e_0\otimes e'_0),\\
d_1(e_0\otimes e'_1)&=e_0\otimes \tilde{d_1}(e'_1)\\
&=e_0\otimes ye'_0 - e_0\otimes e'_0y=ye_0\otimes e'_0-e_0\otimes e'_0y\\
&=(y\otimes 1-1\otimes y)(e_0\otimes e'_0),\\
d_2(e_1\otimes e'_1)&=\tilde{d_1}(e_1)\otimes e'_1-e_1\otimes \tilde{d_1}(e'_1)\\
&=xe_0\otimes e'_1 - e_0x\otimes e'_1-e_1\otimes ye'_0 + e_1\otimes e'_0y \\
&=xe_0\otimes e'_1-e_0\otimes e'_1x-ye_1\otimes e'_0+xe_1\otimes e'_0+e_1x\otimes e'_0+e_1\otimes e'_0y\\
&=xe_0\otimes e'_1-e_0\otimes e'_1x-ye_1\otimes e'_0+xe_1\otimes e'_0+e_1\otimes e'_0x+e_1\otimes e'_0y\\
&=(x\otimes 1-1\otimes x)(e_0\otimes e'_1)+(x\otimes 1+1\otimes x-y\otimes 1+1\otimes y)(e_1\otimes e'_0).
\end{align*}

We next find expressions for Hochschild cocycles on which to use
the techniques of Section~\ref{sec:Gbracket-techniques}:
Apply $\text{Hom}_{(A\otimes_{\tau}B)^e}(-,A\otimes_{\tau}B)$ 
to sequence~(\ref{total}).
Since each $P_i\ot Q_j$ is isomorphic to the free
$(A\otimes_{\tau}B)^e$-module $(A\otimes_{\tau}B)^e$, we find that
\[
  {\rm{Hom}}_{(A\otimes_{\tau}B)^e} (P_i\ot Q_j, A\ot_{\tau}B)
  \cong A\ot_{\tau} B
\]
for each $i,j$, and thus the resulting complex becomes 
\begin{equation}\label{hom total}
\xymatrix{
 0 \ar[r]& A\otimes_{\tau}B \ar[r]^-{d^{*}_1}& 
 (A\otimes_{\tau}B)\oplus (A\otimes_{\tau}B) \ar[r]^-{d^{*}_2}& 
 A\otimes_{\tau}B\ar[r]& 0 . }
\end{equation}
By definition, we have
\begin{equation}\label{HHJ}
\HH^*(A\otimes_{\tau}B)=\Ker(d^{*}_1)\ \oplus \ 
\Ker(d^{*}_2)/\Img(d^{*}_1)\ \oplus \  (A\otimes_{\tau}B)/\Img(d^{*}_2) .
\end{equation}

From now on, we assume that the characteristic of $k$ is 0.
The cohomology in positive characteristic can also be found from
this complex, but it will be different. 

By~\cite[Theorem 2.2]{Sh}, since char$(k)=0$, 
the center of $A\ot_{\tau} B$ is $Z(A\ot_{\tau}B) \cong k$.
Therefore $\HH^0(A\ot_{\tau}B)\cong k$, which is precisely $\Ker (d_1^*)$. 

Now we describe $\HH^1(A\ot_{\tau}B)$. 
Recall that $\HH^1(A\ot_{\tau}B)$ is isomorphic to the space of
derivations of $A\otimes_{\tau} B$ modulo inner derivations, called outer derivations. The following theorem, which is written in a slightly different way and proven in~\cite[Theorem 4.6]{Sh}, directly provides outer derivations of $A\ot_{\tau}B$. 
In our notation, 
viewing a derivation $\partial$ as a Hochschild 1-cocycle on our resolution
(\ref{total}), it will take $e_0\ot e_1'$ to $\partial(y)$
and $e_1\ot e_0'$ to $\partial (x)$. 
See also~\cite{BLO} or~\cite{LS} for a more general setting. 

\begin{theo}\label{theo:bracket on degree 1}
	If char$(k)=0$, then each derivation $\partial$ of $A\ot_{\tau}B$ can
	be represented in the form $\partial(y)=\alpha x +p +\text{ad }w (y)$, $\partial(x)=p'x+\text{ad } w(x)$, where $\alpha \in k$, $p \in k[y]$, $p'$ is the derivative of $p$ with respect to $y$ in the usual sense, $w \in A\ot_{\tau}B$, and  $\text{ad }w(\lambda)=w\lambda-\lambda w \text{ for } \lambda \in A\ot_{\tau}B$.
\end{theo}

As a consequence of the theorem, $\HH^1(A\ot_{\tau}B)\cong k\oplus k[y]$,
which can also be shown directly from complex~(\ref{hom total}).

Lastly, we describe $\HH^2(A\ot_{\tau}B)$.
Calculations show that the image of $d_2^*$ in $A\ot_{\tau} B$
is the ideal generated by $x$, so 
\[
\HH^2(A\ot_{\tau}B)=(A\ot_{\tau}B)/\Img(d^*_2)\cong k[y] .
\] 
Finally, (\ref{HHJ}) becomes
\begin{align*}
\HH^*(A\otimes_{\tau}B)& \cong k\ \oplus \  (k\oplus k[y]) \ \oplus \ k[y] ,
\end{align*}
where the first copy of $k$ is in homological degree~0, the middle
two summands $k\oplus k[y]$ are in degree~1, and the last copy of $k[y]$
is in degree~2. (Cf.\ \cite[Corollary 3.11]{LS}.) 

We will next find a diagonal map 
$\Delta: P_{\bu}\ot_{\tau}Q_{\bu}\rightarrow 
(P_{\bu}\ot_{\tau}Q_{\bu}) \ot_{A\ot_{\tau}B} (P_{\bu}\ot_{\tau} Q_{\bu})$ 
for use in computing Gerstenhaber brackets
on $\HH^*(A\ot_{\tau}B)$.
In order to do this, we will need to consider
the total  complex $Q_{\bu}\otimes_{\tau^{-1}} P_{\bu}$ and the differentials $\hat{d}_1, \hat{d}_2$:
\begin{equation*}
\xymatrix{Q_{\bu}\otimes_{\tau^{-1}} P_{\bu} : &&&&\\
 0 \ar[r] & Q_1 \otimes P_1 \ar[r]^-{\hat{d}_2} & 
   (Q_1 \otimes P_0)\oplus (Q_0 \otimes P_1) \ar[r]^-{\hat{d}_1} 
  &  Q_0 \otimes P_0 \ar[r] & 0 }
\end{equation*}
(Recall that $A\ot_{\tau}B\cong B\ot_{\tau^{-1}}A$ as algebras,
and we define $Q_{\bu}\ot_{\tau^{-1}}P_{\bu}$ as a projective 
resolution of $(B\ot_{\tau^{-1}}A)^e$-modules, 
equivalently of $(A\ot_{\tau}B)^e$-modules,
via the techniques of Section~\ref{sec:twisted}.)
We find the differentials $\hat{d}_2$, $\hat{d}_1$ in the same way as we found the differentials for
$P_{\bu}\ot_{\tau}Q_{\bu}$:
\begin{align*}
 \hat{d}_1(e'_1\otimes e_0) &= ye'_0\otimes e_0-e'_0\otimes e_0y ,\\
 \hat{d}_1(e'_0\otimes e_1) &= xe'_0\otimes e_0-e'_0\otimes e_0x ,\\
 \hat{d}_2(e'_1\otimes e_1) &= (1\otimes x-x\otimes 1)
 (e'_1\otimes e_0)+(y\otimes 1-1\otimes y - x\otimes 1 - 1\otimes x)
(e'_0\otimes e_1).
\end{align*}

Now, consider the complexes $Q_{\bu}\otimes_{\tau^{-1}} P_{\bu}$ 
and $P_{\bu}\otimes_{\tau} Q_{\bu}$.
We wish to define
$(A\ot_{\tau}B)^e$-module homomorphisms $\tau_0, \tau_1$ and $\tau_2$
so that $\tau_{\bu}$ is a chain map from $Q_{\bu}\ot_{\tau^{-1}} P_{\bu}$ to $P_{\bu}\ot_{\tau} Q_{\bu}$. Calculations show that the following values define such a chain map
\begin{equation*}
\tau_0(e_0'\otimes e_0)= e_0\otimes e_0', \ \
\tau_1(e_1'\otimes e_0)= e_0\otimes e_1', \ \ 
\tau_1(e_0'\otimes e_1) = e_1\otimes e_0', \ \
\tau_2(e_1'\otimes e_1)= -e_1\otimes e_1'.
\end{equation*}
We may now set $\sigma'_{\bu} = 1_P\ot \tau_{\bu}\ot 1_Q$ 
which induces a map 
\[
   \sigma_{\bu} :  (P_{\bu}\ot_{\tau}Q_{\bu})\ot _{A\ot_{\tau}B} 
   (P_{\bu}\ot_{\tau}Q_{\bu})
   \rightarrow (P_{\bu}\ot_A P_{\bu})\ot_{\tau}(Q_{\bu}\ot_BQ_{\bu}),
\]
similar to the setting in Lemma~\ref{a}.
We may check that condition~(\ref{eqn:sigma-condn}) holds. 

A diagonal map
$\Delta_P$ on $P_{\bu}$ is given by 
$\Delta_P(e_0) = e_0\otimes e_0$,  
$\ \Delta_P(e_1)= e_0\otimes e_1 + e_1\otimes e_0$,
and similarly $\Delta_Q$ on $Q_{\bu}$.
(These are the maps given by viewing $P_{\bu}$, $Q_{\bu}$
as subcomplexes of bar complexes.) 

We define a diagonal map $\Delta$ on $P_{\bu}\ot_{\tau}Q_{\bu}$, 
taking $\Delta$ to be the
composition of the following maps in each degree: 
$$P_{\bu}\ot Q_{\bu}\stackrel{\Delta_P'\otimes\Delta_Q'}{\relbar\joinrel\relbar\joinrel\relbar\joinrel\longrightarrow} (P_{\bu}\ot P_{\bu}) \ot (Q_{\bu} \ot Q_{\bu})\stackrel{1_P\otimes \tau^{-1}_{\bu} \otimes 1_Q}{\relbar\joinrel\relbar\joinrel\relbar\joinrel\relbar\joinrel\relbar\joinrel\longrightarrow}
P_{\bu}\ot Q_{\bu}\ot P_{\bu}\ot Q_{\bu} \rightarrow
 (P_{\bu}\ot Q_{\bu}) \ot_{A\ot_{\tau}B} (P_{\bu}\ot  Q_{\bu}) , $$
where the last map is a quotient map,
and by $\Delta_P'$ we mean the map to $P_{\bu}\ot P_{\bu}$ defined by
an analogous formula as that of the diagonal map to $P_{\bu}\ot_A P_{\bu}$, and similarly $\Delta_Q'$. 
By construction, $\Delta$ is an $(A\ot_{\tau} B)^e$-module homomorphism
and a chain map. 
Under this composition of maps, we find that 
\begin{align*}
\Delta(e_0\ot e_0') &= (e_0\ot e_0')\ot (e_0\ot e_0'),\\
\Delta (e_1\otimes e'_0) &= (e_0\otimes e'_0)\otimes (e_1\otimes e'_0) + (e_1\otimes e'_0)\otimes (e_0\otimes e'_0) , \\
\Delta (e_0\otimes e'_1) &= (e_0\otimes e'_1)\otimes (e_0\otimes e'_0) + (e_0\otimes e'_0)\otimes (e_0\otimes e'_1),\\
\Delta (e_1\otimes e'_1) &= 
(e_0\otimes e'_0)\otimes (e_1\otimes e'_1) - (e_0\otimes e'_1)\otimes (e_1\otimes e'_0)\\
& \quad +(e_1\otimes e'_0)\otimes (e_0\otimes e'_1)+ (e_1\otimes e'_1)\otimes (e_0\otimes e'_0).
\end{align*}
Direct calculations show that this diagonal map $\Delta$
makes $P_{\bu}\ot_{\tau} Q_{\bu}$ a counital differential graded coalgebra,
so the hypothesis of Theorem~\ref{thm:bracket} holds. 
We may now apply Theorem~\ref{mainthm} 
to calculate Gerstenhaber brackets. By formulas~(\ref{eqn:psi-f}) and~(\ref{3}), 
$\psi_f=\phi(1\ot f\ot 1)\Delta^{(2)}$ and 
$\phi=(\phi_P\ot \mu_Q\ot 1_Q+1_P\ot \mu_P\ot \phi_Q) \sigma$.
We will use formulas for $\phi_P$ and $\phi_Q$ from~\cite[Section 4]{NW}: 
$\phi_P(e_0 \ot x^t e_0)=\sum_{i=0}^{t-1}x^i e_1 x^{t-i-1}$ and 
similarly $\phi_Q$.
These values of $\phi$ and $\Delta$ may be used to calculate
the Gerstenhaber bracket of any two Hochschild cocycles via
Theorem~\ref{thm:bracket}. 
We do one such calculation to explain the technique. 

Let $f\in \HH^1(A\ot_{\tau}B)\cong k \oplus  k[y]$ and $g\in \HH^2(A\ot_{\tau}B)\cong  k[y]$. 
We take $f$ to correspond to $y$ and $g$ to $y^3$,
so that 
\[
   f(e_0\otimes e'_1)=y, \ \ \ f(e_1\otimes e'_0)=x, 
   \ \ \ g(e_1\otimes e'_1)=y^3.
\]
We calculate the value of the bracket $[f,g]$ on $e_1\ot e'_1$:
\begin{align*}
f\psi_g(e_1\ot e'_1)
&=f \phi (1\ot g\ot 1)(\Delta\ot 1)(\Delta(e_1\ot e'_1))\\
&=f \phi (1\ot g\ot 1)(\Delta\ot 1)((e_0\ot e'_0)\ot (e_1\ot e'_1)- (e_0\ot e'_1)\ot (e_1\ot e'_0)\\
&\quad + (e_1\ot e'_0)\ot (e_0\ot e'_1) + (e_1\ot e'_1)\ot (e_0\ot e'_0))\\
&=f \phi((1\ot g\ot 1)
[(e_0\ot e'_0)\ot (e_0\ot e'_0)\ot (e_1\ot e'_1)\\
&\quad -(e_0\ot e'_1)\ot (e_0\ot e'_0)\ot (e_1\ot e'_0)
 -(e_0\ot e'_0)\ot (e_0\ot e'_1)\ot (e_1\ot e'_0)\\
&\quad +(e_0\ot e'_0)\ot (e_1\ot e'_0)\ot (e_0\ot e'_1)
 +(e_1\ot e'_0)\ot (e_0\ot e'_0)\ot (e_0\ot e'_1)\\
&\quad +(e_0\ot e'_0)\ot (e_1\ot e'_1)\ot (e_0\ot e'_0) 
-(e_0\ot e'_1)\ot (e_1\ot e'_0)\ot (e_0\ot e'_0)\\
&\quad +(e_1\ot e'_0)\ot (e_0\ot e'_1)\ot (e_0\ot e'_0) 
+(e_1\ot e'_1)\ot (e_0\ot e'_0)\ot (e_0\ot e'_0)]\\
&=f \phi(e_0\ot e_0'y^3 \ot e_0\ot e_0')\\
&=f ((\phi_P \ot \mu_Q\ot 1_Q)+ (1_P\ot\mu_P\ot \phi_Q)) 
  (e_0\ot e_0 \ot e_0' y^3\ot e_0')\\
&=f(y^2e_0\ot e_1' + ye_0\ot e_1' y + e_0\ot e_1'y^2)\\
&=y^3+y^3+y^3 \ \ = \ \ 3y^3 ,
\end{align*}
and
\begin{align*}
g\psi_f(e_1\ot e'_1)
&=g \phi (1\ot f\ot 1)(\Delta\ot 1)(\Delta(e_1\ot e'_1))\\
&=g \phi (1\ot f\ot 1)(\Delta\ot 1)((e_0\ot e'_0)\ot (e_1\ot e'_1)- (e_0\ot e'_1)\ot (e_1\ot e'_0) \\
&\quad + (e_1\ot e'_0)\ot (e_0\ot e'_1) + (e_1\ot e'_1)\ot (e_0\ot e'_0))\\
&=g \phi((1\ot f\ot 1)
[(e_0\ot e'_0)\ot (e_0\ot e'_0)\ot (e_1\ot e'_1)\\
& \quad -(e_0\ot e'_1)\ot (e_0\ot e'_0)\ot (e_1\ot e'_0
 -(e_0\ot e'_0)\ot (e_0\ot e'_1)\ot (e_1\ot e'_0)\\
&\quad +(e_0\ot e'_0)\ot (e_1\ot e'_0)\ot (e_0\ot e'_1)
 +(e_1\ot e'_0)\ot (e_0\ot e'_0)\ot (e_0\ot e'_1)\\
&\quad +(e_0\ot e'_0)\ot (e_1\ot e'_1)\ot (e_0\ot e'_0)
 -(e_0\ot e'_1)\ot (e_1\ot e'_0)\ot (e_0\ot e'_0)\\
& \quad +(e_1\ot e'_0)\ot (e_0\ot e'_1)\ot (e_0\ot e'_0)
+(e_1\ot e'_1)\ot (e_0\ot e'_0)\ot (e_0\ot e'_0)]\\
&=g \phi(-e_0\ot e_0'y \ot e_1\ot e_0' + e_0\ot e_0'\ot xe_0\ot e_1' \\
&\quad - e_0\ot e_1'\ot xe_0\ot e_0'
 + e_1\ot e_0'y \ot e_0\ot e_0') \\
&=g ((\phi_P \ot \mu_Q\ot 1_Q)+ (1_P\ot \mu_P\ot \phi_Q))[ - e_0\ot e_1 \ot e_0'y \ot e_0' - xe_0\ot e_1 \ot e_0' \ot e_0'\\
&\quad  - e_0\ot e_1x \ot e_0' \ot e_0' 
 + e_0\ot xe_0 \ot e_0' \ot e_1'  - e_0\ot xe_0 \ot e_1' \ot e_0' + e_1\ot e_0 \ot e_0'y \ot e_0' ]\\
&=g(e_1\ot e_1' + e_1\ot e_1') \ \ = \ \ 2y^3 .
\end{align*}
Hence, $[f,g](e_1\ot e'_1)=3y^3-2y^3=y^3$, and we have $[f,g]= g$. 

\begin{rema}
{\em
In the notation of~\cite[Theorem 6.6]{LS}, take $h=x^2$, $n=1$,
$\pi_h=x$, and $a_1=xy$. 
Our $f$ is then their ad$_{a_1}$, and our bracket calculation agrees with theirs.
}\end{rema}

\end{document}